\documentclass[12pt]{article}
\usepackage{amssymb}
\usepackage{amsmath}
\usepackage{mathabx}
\usepackage{amsthm}
\usepackage{amsmath,stmaryrd}
\usepackage{tikz-cd}
\usepackage{mathabx}
\usepackage{mathrsfs}
\usepackage{graphicx}
\usepackage{todonotes}
\usepackage{hyperref}
\usepackage[capitalize]{cleveref}
\usepackage{marginnote}
\usepackage{thmtools}
\usepackage[backend=biber, url=false, maxcitenames=10, maxbibnames=10]{biblatex}
\declaretheorem[numberwithin=section,name=Theorem]{thm}

\declaretheorem[name=Definition, style=definition, sibling=thm]{defn}
\declaretheorem[sibling=thm]{lem}
\declaretheorem[sibling=thm]{proposition}
\declaretheorem[sibling=thm]{question}

\newcommand{\A}{{\mathcal{A}}}
\newcommand{\B}{{\mathcal{B}}}
\newcommand{\Z}{{\mathbb{Z}}}

\renewcommand{\bar}[1]{\overline{#1}}

\renewcommand{\phi}{\varphi}

\newbox\gnBoxA

\newdimen\gnCornerHgt
\setbox\gnBoxA=\hbox{$\ulcorner$}
\global\gnCornerHgt=\ht\gnBoxA
\newdimen\gnArgHgt
\def\Godelnum #1{%
\setbox\gnBoxA=\hbox{$#1$}%
\gnArgHgt=\ht\gnBoxA%
\ifnum \gnArgHgt<\gnCornerHgt \gnArgHgt=0pt%
\else \advance \gnArgHgt by -\gnCornerHgt%
\fi \raise\gnArgHgt\hbox{$\ulcorner$} \box\gnBoxA %
\raise\gnArgHgt\hbox{$\urcorner$}}

\hypersetup{
    colorlinks=true,
    }

\usepackage[margin=1.5in]{geometry}
\title{Scott Complexity of Reduced Abelian $p$-Groups}
\date{\today}
\author{Rachael Alvir, Barbara F. Csima, Luke MacLean}

\begin{document}

\maketitle

\begin{abstract}

Given a reduced abelian $p$-group, we give an upper bound on the Scott complexity of the group in terms of its Ulm invariants. For limit ordinals, we show that this upper bound is tight. This gives an explicit sequence of such groups with arbitrarily high Scott complexity below $\omega_1$. Along the way, we give a largely algebraic characterization of the back-and-forth relations on reduced abelian $p$-groups, making progress on an open problem of Ash and Knight's from \cite{AK2000}.

\end{abstract}

\section{Introduction}

Scott's isomorphism theorem \cite{Scott1965} states that any countable structure can be described $\omega$-categorically by a single sentence $\phi$ of $L_{\omega_1 \omega}$, known as the \emph{Scott sentence} of that structure. In the proof of this result, there arises an interesting invariant known as the \emph{Scott rank}, which contains a wealth of metamathematical information about the structure. Unfortunately for many years non-equivalent definitions of Scott rank were considered in the literature. More recently, Antonio Montalb\'an \cite{Montalbn2015} argued to standardize the definition of a structure's \emph{categoricity Scott rank:} the least ordinal $\alpha$ such that the structure has a $\Pi_{\alpha + 1}$ Scott sentence. Though a useful notion, there is a little arbitrariness in Montalb\'an's definition as Scott sentences of other syntactic forms can be considered. It is argued in Alvir, Greenberg, Harrison-Trainor, and Turetsky \cite{AGHTT2021} that we should consider the \emph{Scott complexity} of a structure: the least complexity ($\Pi_{\alpha}, \Sigma_\alpha$, or $d$-$\Sigma_\alpha$) of a Scott sentence for that structure. Here, a formula of $L_{\omega_1 \omega}$ is said to be $d$-$\Sigma_\alpha$ if it is a conjunction of a $\Sigma_\alpha$ and a $\Pi_\alpha$ formula.

In this paper, we investigate the Scott complexity of reduced abelian $p$-groups. In many cases we calculate the Scott complexity of a given group exactly, and in all other cases we give an upper bound. In particular, we show in Theorem \ref{ScottSentences} that if $\A$ has length $\lambda_\A = \omega \cdot \alpha$ where $\alpha$ is a limit ordinal, then $\A$ has Scott complexity $\Pi_{2 \alpha + 1}$. That is, we give an explicit sequence of structures with arbitrarily high Scott complexity. Similar results about structures with arbitrarily high Scott complexity were obtained for scattered linear orders in Alvir and Rossegger \cite{AR2020}.
Optimal complexity Scott sentences of groups in particular are considered in Knight and Saraph \cite{knight2018scott}, Ho \cite{ho2017describing}, and Harrison-Trainor and Ho \cite{harrison2018optimal}. 

To obtain our results, we first characterize the back-and-forth relations on reduced abelian $p$-groups. The standard back-and-forth relations on arbitrary structures can be defined as follows \cite{AK00}. 

\begin{defn}
    Let $\A$ and $\B$ be structures, with finite tuples $\bar a \in \A$ and $\bar b \in \B$. Then: 
    \begin{itemize}
        \item $(\A, \bar a) \leq_0 (\B, \bar b)$ iff $\bar a$ and $\bar b$ satisfy the same quantifier-free formulas.\footnote{Note that for certain applications when the language is infinite, the defintion of $\leq_0$ may vary from the one we use here.}
        \item For $\alpha >0$, $(\A, \bar a) \leq_\alpha (\B, \bar b)$ if, for each $\beta < \alpha$ and for each $\bar d \in \B$, there is a $\bar c \in \A$ such that $(\B, \bar b \bar d) \leq_\beta (\A, \bar a \bar c)$.
    \end{itemize}
\end{defn}

Giving a characterization of the back-and-forth relations within a class of structures is an interesting result in its own right with a wide array of possible uses. For example, in addition to results about Scott complexity, they are often used to apply general results on intrinsically $\Sigma_\alpha$ relations and on $\Delta_\alpha$ categoricity and stability. Our characterization also makes progress on a long-standing open problem from Ash and Knight's book \cite{AK2000}. There, it was asked which algebraic conditions characterize the back-and-forth relations on reduced abelian $p$-groups. The characterization we give is largely algebraic. 

There is a connection between Scott sentences and descriptive set theory. For example, a result of Lopez-Escobar with several strengthenings due to Vaught and D. Miller states that a structure $A$ has a $\Pi_{\alpha + 1}$ Scott sentence iff the set of presentations of $A$ is boldface $\Pi_{\alpha + 1}$ in the Borel hierarchy; see the discussion in Montalban's paper \cite{Montalbn2015}. It is also known that for a class of countable structures $K$, the isomorphism relation on $K$ is not a Borel equivalence relation iff structures in that class attain arbitrarily high Scott complexity below $\omega_1$. There are several common classes of groups for which is it known that the isomorphism relation is not a Borel equivalence relation; however, it is not usually known exactly which structures of that class achieve arbitrarily high Scott complexity and why, which is one thing our paper accomplishes. In Friedman and Stanley \cite{Friedman_Stanley_1989}, for example, it is shown that isomorphism as an equivalence relation on the class of abelian $p$-groups is complete analytic and thus not Borel. This is accomplished by giving a reduction from pseudo-well-orders to $p$-groups; thus to find an explicit description of abelian $p$-groups that achieve arbitrarily high Scott complexity via this reduction, it would be necessary to do so first for pseudo-well-orders. In addition, the reduction given in Friedman and Stanley's paper is accomplished by building an abelian $p$-group from a tree via a construction of Feferman's, a construction which is omitted. Thus, for many reasons the explicit construction of abelian $p$-groups for each prime $p$ with arbitrarily high Scott complexity, and the direct reason why the Scott complexity must increase, is obscured. Our results show, roughly, that Scott complexity increases with the length of the group at limit ordinals. In addition our results improve those in \cite{Friedman_Stanley_1989} slightly, showing that for each prime $p$, isomorphism restricted to the class of \emph{reduced} abelian $p$-groups is still not Borel. 

\section{Preliminaries}
For a standard resource covering the basics of reduced abelian $p$-groups and Ulm's Theorem, see \cite{Kaplansky1954}. Throughout this paper, all groups are assumed to be countable abelian $p$-groups. For a fixed prime $p$, a group $G$ is said to be primary (or $p$-primary) if for each group element $x \in G$ $p^nx = 0$ for some $n\geq1$. A group $G$ is divisible if, for every $x \in G$, there exists a $y \in G$ and $n \in \omega$ such that $x = ny$. A reduced group is one with no non-trivial divisible subgroups.

Let $G_0=G$, and for any ordinal $\alpha$ we define $G_{\alpha+1} = pG_\alpha = \{x \in G : \exists y \in G_{\alpha} \, py = x\}$. If $\alpha$ is a limit ordinal, then we define $G_{\alpha} = \bigcap_{\beta<\alpha} G_\beta$. Notice that if $\alpha < \beta$ then $G_\alpha \supseteq G_\beta$. We define the largest $\alpha$ such that $x \in G_\alpha$ to be the \emph{height} of $x$ and denote it $h(x)$. This means that $G_\alpha$ is the set of elements of height at least $\alpha$. The height function does not always play nicely with respect to sums. For $x,y \in G$ if $h(x) < h(y)$ then $h(x+y) = h(x)$, but if $h(x) = h(y)$ then $h(x+y) \geq h(x)$.  Given a subgroup $S$ of $G$, we say that $x$ is proper with respect to $S$ if $h(x) \geq h(x+s)$ for all $s \in S$. This is the same as the height of $x$ being maximal in its coset $mod \,S$. For such proper elements, $h(x+s) = \min\{h(x),h(s)\}$. We also know that for all $x \neq 0$ we have $h(px) > h(x)$. We define $h(0) = \infty$ to be greater than any ordinal to ensure this property.

We similarly define, for all $\alpha$, $P_{\alpha} = P \cap G_{\alpha}$, where $P = \{ g \in G : pg=0\}$. Viewing $P_{\alpha}/P_{\alpha+1}$ as a vector space over $\mathbb{Z}_p$, we define the $\alpha$-th Ulm invariant to be its dimension, and denote it $f_G(\alpha)$. In addition, since $G$ is countable, there must be an ordinal $\lambda$ called the length of $G$ such that $G_{\lambda} = G_{\xi}$ for all $\xi\geq \lambda$. Since $G$ is reduced, we can conclude that $G_\lambda=\{0\}$.

Our descriptions of reduced abelian $p$-groups up to isomorphism rely heavily on the following important theorem: 

\begin{thm}[Ulm \cite{Kaplansky1954}]
    Two reduced countable primary abelian groups are isomorphic if and only if they have the same Ulm invariants.
\end{thm}


Thus, to describe a given a group $G$ of length $\lambda$ up to isomorphism, we wish to say for each $\alpha < \lambda$ that there is a basis of $P_\alpha / P_{\alpha+1}$ of size $k$. 

Barker \cite{Barker1995} has shown that

\begin{proposition}\label{Galpha}
    Let $\alpha$ be an ordinal. The group $G_{\omega \cdot \alpha}$ is definable by a $\Pi^{0}_{2\alpha}$ formula. The group $G_{\omega \cdot \alpha +n}$ for $n>0$ is definable by a $\Sigma^{0}_{2\alpha +1}$ formula.
\end{proposition}

Given an $x \in G_\alpha$ we have that $h(px) > h(x)$. Note that it is not always the case that $h(px)=h(x)+1$, as there are some $x$ for which the height of $px$ increases more. Let $p^{-1}G_{\alpha+2}= \{z \in G : pz \in G_{\alpha+2}\}$. Then, given a subgroup $S$, we define $S_\alpha = S \cap G_\alpha$ and look at the subgroup $S_\alpha^* = S_\alpha \cap p^{-1}G_{\alpha+2}$. For any subgroup $S$ of $G$, it turns out that $S_{\alpha}^*/S_{\alpha + 1}$ is isomorphic to a subgroup of $P_{\alpha} / P_{\alpha + 1}$.

Let us look at the map that does this. First, let us try to find a map from $S_{\alpha}^*$ to $P_\alpha$. Now, for any $x \in S_{\alpha}^*$ $px \in G_{\alpha + 2} = pG_{\alpha +1}$ so $px = py$ for some $y \in G_{\alpha+1}.$ In particular, $p(x-y)$ = 0, so the element $x-y$ has order $p$. Since $x$ is in $S_{\alpha}$, and $y$ is in $G_{\alpha + 1}$ and thus $G_{\alpha}$, we have that $x-y$ is in $G_{\alpha}$ (since this is a subgroup.) Thus $x-y \in P_{\alpha}$. Now the ``map'' that sends $x$ to $x-y$ is not actually well-defined, since $y$ is not unique. However, the map from $S_{\alpha}^*$ to $P_{\alpha} / P_{\alpha + 1}$ which sends $x$ to $\bar{x-y}$ (where $\bar{x-y}$ is the image of $x-y$ under the natural homomorphism) is actually well-defined. Let us call this map $u$. It turns out that $u$ is a homomorphism whose kernel is exactly $S_{\alpha+1}$. This induces an embedding, which we will also call $u$, that shows that $S_{\alpha}^*/S_{\alpha + 1}$ is isomorphic to a subgroup of $P_{\alpha} / P_{\alpha + 1}$. 

The usefulness of this map lies in the following lemma.

\begin{lem}[Kaplansky \cite{Kaplansky1954}]\label{Ulmiso}
    Let $S$ be a subgroup of the group $G$ and $u$ as defined above. Then the following are equivalent:
    \begin{enumerate}
        \item The range of $u$ is not all of $P_{\alpha}/P_{\alpha+1}$.
        \item There exists an $x \in P_\alpha$ proper with respect to $S$.
    \end{enumerate}
\end{lem}

\section{Upper Bounds on Scott Complexity}

Our first theorem gives a Scott sentence for an arbitrary reduced abelian $p$-group, depending on the length of the group and whether certain Ulm invariants are infinite or finite. This yields an upper bound on the Scott complexity of the group. 

In this paper, when writing formulas note that we do not have the natural numbers as part of our language. Given a group $G$, when we write $n \cdot x$ for $n \in \omega, x \in G$, we mean $x + \overset{n}{\cdots} + x$.

\begin{thm}\label{ScottSentences}
    Let $G$ be a group with length $\lambda >0$ and Ulm invariants given by $f$. Then for $0<n<\omega$,
    \begin{itemize}
        \item[(i)] If $\lambda = \omega \cdot \gamma +n$ and $f(\omega \cdot \gamma+k)<\infty$ for $0\leq k < n$ then $G$ has a Scott Sentence of complexity $d$-$\Sigma_{2\gamma +2}$.
        \item[(ii)] If $\lambda = \omega \cdot \gamma +n$ and $f(\omega \cdot \gamma +i)=\infty$ for some $i$ with $0 \leq i < n$ then $G$ has a Scott Sentence of complexity $\Pi_{2\gamma +3}$.
        \item[(iii)] If $\lambda = \omega \cdot \gamma$ then $G$ has a Scott Sentence of complexity $\Pi_{2\gamma+1}$.
    \end{itemize}
\end{thm}

We will show at the end of Section \ref{ScottComplexitySection} in Theorem \ref{Scottcomplex} that for every group in case (iii) above, the Scott sentence we give here is best possible. We also give a class of groups in case (ii) for which the the Scott sentences here are best possible. However, we show in Lemma \ref{NotBestPossible} that such groups can have simpler Scott sentences. It is not known if the sentences in case (i) are best possible or not.

\begin{proof}

Notice that, for an element $x \in G$, $x \in P$ iff $p\cdot x = 0$. Thus, $P_{\alpha}$ is definable with the same complexity as $G_{\alpha}$. Letting $\varphi_{\beta}$, $\psi_{\beta}$ define $G_{\beta}$,$P_{\beta}$ respectively, the following says $P_{\alpha} / P_{\alpha + 1}$ has dimension at least $m$ for $m\in\omega$, by saying that there exists $m$ linearly independent elements:

\begin{align*}
\hat{\theta}_{\alpha}^m \equiv \exists g_1, \ldots , g_m\left(\bigwedge_{i\leq m} \psi_{\alpha} (g_i) \land \bigwedge_{(n_1, \ldots, n_m) \in \mathbb{Z}_p \setminus \{\vec{0}\}} \neg \psi_{\alpha + 1}(n_1 \cdot g_1 + \cdots + n_m \cdot g_m) \right) 
\end{align*}






Note that if $\alpha = \omega \cdot \beta + r$ then $\psi_{\alpha}, \psi_{\alpha + 1}$ both have complexity $\Sigma_{2 \beta + 1}$. Thus the above formula has complexity $\Sigma_{2 \beta + 2}$.

Next, note that the formula $$\theta_{\alpha}^m \equiv \hat{\theta}_{\alpha}^m \land \neg \hat{\theta}_{\alpha}^{m+1}$$ says that $P_{\alpha}/P_{\alpha+1}$ has dimension exactly $m$. This formula has complexity $d$-$\Sigma_{2 \beta + 2}$.

If $P_{\alpha}/P_{\alpha+1}$ has dimension $\infty$ over $\mathbb{Z}_p$ then we must say $$\theta_\alpha^\infty \equiv \bigwedge_{m \in \omega} \hat{\theta}_\alpha^m$$

which has complexity $\Pi_{2 \beta + 3}$.

We also need to know the length $\lambda$ of the group $G$, so that we can distinguish it from a group $G'$ of longer length, but with all the same Ulm invariants up to $\lambda$ as $G$. The length of the group is the first ordinal $\lambda$ such that $G_\lambda=0$, so a formula saying this fact is
\begin{equation*}
    L_\lambda \equiv \forall x (\varphi_\lambda(x) \rightarrow x=0) \wedge \bigwedge_{\alpha < \lambda} \exists x (\varphi_\alpha(x) \wedge x\neq 0).
\end{equation*}

If $\lambda = \omega \gamma$ is a limit ordinal, then this formula is $\Pi_{2 \gamma + 1}$. If $\lambda = \omega \gamma + n$ for $n > 0$ then this formula has complexity $d$-$\Sigma_{2 \gamma +1}$.

Thus, the group $G$ with length $\lambda$ and Ulm invariants given by $f$ has Scott sentence
$$
  \Phi_G \equiv L_\lambda \wedge \bigwedge_{\alpha < \lambda} \theta_\alpha^{f(\alpha)} \wedge \Omega
$$
where $\Omega$ is a $\Pi^0_2$ sentence saying that $G$ is an Abelian $p$-group. Any structure that is isomorphic to $G$ will trivially satisfy this formula, and if a group satisfies this formula, then Ulm's Theorem tells us that it is isomorphic to $G$.

We now consider the complexity of the formula $\Phi_G.$

\emph{Case (i):} Suppose that $\lambda = \omega \cdot \gamma +n$ and $f(\omega \cdot \gamma +k)<\infty$ for $0\leq k < n$. Then $\theta_\alpha^{f(\alpha)}$ for $\omega \cdot \gamma \leq \alpha < \lambda$ has complexity $d$-$\Sigma_{2 \gamma + 2}$. For $\alpha < \omega \cdot \gamma$, $\theta_\alpha^{f(\alpha)}$ is at most $\Pi_{2 \gamma' + 3}$ for $\gamma' < \gamma$. Thus $\Phi_G$ has complexity $d$-$\Sigma_{2 \gamma + 2}$. 

\emph{Case (ii):} Suppose that $\lambda = \omega \cdot \gamma +n$ and $f(\omega \cdot \gamma +i)=\infty$ for some $i$ with $0 \leq i < n$. Then $\theta_{\omega \cdot \gamma + i}^{f(\omega \cdot \gamma + i)}$ has complexity $\Pi_{2 \gamma + 3}$. For all other $\alpha$, $\theta_\alpha^{f(\alpha)}$ has the same complexity if infinite and lower complexity otherwise. Thus $\Phi_G$ has complexity $\Pi_{2 \gamma + 3}$.

\emph{Case (iii):} Suppose that $\lambda = \omega \cdot \gamma$. For $\alpha < \omega \cdot \gamma$, $\theta_\alpha^{f(\alpha)}$ is at most $\Pi_{2 \gamma' + 3}$ for $\gamma' < \gamma$. Thus $\bigwedge_{\alpha < \lambda} \theta_\alpha^{f(\alpha)}$ is $\Pi_{2 \gamma}$. However, $L_\lambda$ has complexity $\Pi_{2 \gamma + 1}$. Thus $\Phi_G$ has complexity $\Pi_{2 \gamma + 1}$.
\end{proof}

The following example shows that when $G$ is an infinite group with one Ulm invariant, the Scott sentence given above may not be best possible. This is because there may be values other than the Ulm invariants that can determine the isomorphism type of the group, and are easier syntactically to describe. In this example, we also compute the Scott complexity of the group exactly. 


\begin{lem}\label{NotBestPossible}
The group $G = (\mathbb{Z}_{p^n})^\omega$ has Scott complexity $\Pi_2$
\end{lem}
\begin{proof}
It is well-known that an infinite structure cannot have a Scott sentence simpler than $\Pi_2$.

We claim that the conjunction of the following statements, along with axioms stating that $G$ is an abelian group, is a Scott sentence for $G$:
\begin{enumerate}
    \item $\forall x \ p^n \cdot x = 0$ 
    
    \item $\forall x \bigvee_{m \in \omega} p^m \cdot x = 0 \land \bigwedge_{k < p^m} k \cdot x \neq 0.$ 
    
    \item $\bigwedge_{m \in \omega} \exists x_1, \ldots, x_m \bigwedge_{i \neq j} x_i \neq x_j \land \bigwedge_{i = 1}^{n} \bigwedge_{k < p^n} k \cdot x_i \neq 0.$

    \item $\forall x \neq 0 (p \cdot x = 0 \rightarrow \exists y \ x = p^{n-1}y.)$
    
\end{enumerate}

Clause (1) says that every element has order at most $p^n$ ($\Pi_1)$. 
Clause (2) says that $G$ is a $p-$group ($\Pi_2$). 
Clause (3) says the group has infinitely many elements of order $p^n$ ($\Pi_2$). 
Clause (4) says that every element of order $p$ has height at least $n-1$. 

Clearly all of these statements are true about $G$. Next, we show that they determine $G$ up to isomorphism. Let $H$ be a group satisfying the conjunction of the above clauses. Then since $H$ is of bounded order and a $p$-group, $H$ is a direct sum of cyclic groups of the form $\mathbb{Z}_{p^m}$ for $m \leq n$. Since $H$ has infinitely many elements of order $p^n$, $H$ must contain infinitely many copies of $\mathbb{Z}_{p^n}$. It remains to show that $H$ contains no group of the form $\mathbb{Z}_{p^m}$ for $m < n$ as a direct summand. However, if it did it would have an element of order $p$ of height less than $n-1$, which is impossible by clause (4). 
\end{proof}

\section{Scott Complexity}\label{ScottComplexitySection}
We saw earlier that when the length of the group $G$ is finite and there is exactly one infinite nonzero Ulm invariant, that the Scott sentence given \cref{ScottSentences} may not be best possible. However, if $G$ has at least two infinite nonzero Ulm invariants, this is no longer the case. 

To prove this, we will need the following results. Recall that $G \preceq_{\Sigma_1} H$ iff $G,H$ agree on all $\Sigma_1$ sentences with parameters from $G$. Eklof \cite{Eklof1972} proved that, for abelian groups $G$ and $H$, $G \preceq_{\Sigma_1} H$ iff $G$ is a pure subgroup of $H$ 
and every $\Sigma_1$ sentence true of $H$ is true of $G$. Alvir proved the following: 

\begin{thm}[Alvir, \cite{Alvir}]\label{Rachael'sThm}
Suppose that $(A_i)_{i \in \omega}$ is a chain such that $A_i \preceq_{\Sigma_\alpha} A_{i + 1}$, and $A = \cup A_i$ where $A$ is not isomorphic to any $A_i$. 
Then $A$ has no $d$-$\Sigma_{\alpha+1}$ Scott sentence.
\end{thm}

Using these results, we can begin to calculate Scott complexity exactly.

\begin{thm}
    Suppose that $G$ is a reduced abelian $p$-group of finite length with at least two infinite nonzero Ulm invariants. Then $G$ has Scott complexity $\Pi_3$. 
\end{thm}

\begin{proof}
We have seen in \cref{ScottSentences} that $G$ has a $\Pi_3$ Scott sentence. It remains to show that $G$ has no $d$-$\Sigma_2$ Scott sentence.

Assuming that $G$ has two infinite Ulm invariants $f(n)=\infty$ and $f(m)=\infty$ for $n<m$ we can write $G$ as $H \oplus (\mathbb{Z}_{p^n})^\omega$ or more importantly as a chain $\bigcup_{i} H \oplus (\mathbb{Z}_{p^n})^i$. We seek to show that at each stage $H \oplus (\mathbb{Z}_{p^n})^i \preceq_{\Sigma_1}  H \oplus (\mathbb{Z}_{p^n})^{i+1}$. It is well-known that if $G$ is a primary group and $H$ a direct summand, then $H$ is a pure subgroup of $G$. Thus we have that $ H \oplus (\mathbb{Z}_{p^n})^i$ is a pure subgroup of $H \oplus (\mathbb{Z}_{p^n})^{i+1}$. By Eklof's result, it remains to show that $ H \oplus (\mathbb{Z}_{p^n})^{i} \leq_1  H \oplus (\mathbb{Z}_{p^n})^{i+1}$. 

This is true because we can embed $\mathbb{Z}_{p^n}\oplus (\mathbb{Z}_{p^m})^\omega$ into $(\mathbb{Z}_{p^m})^\omega$ since $\mathbb{Z}_{p^m}$ contains a subgroup isomorphic to $\mathbb{Z}_{p^n}$. Since $H$ contains $(\Z_{p^m})^\omega$ as a direct summand, we can embed $H \oplus (\mathbb{Z}_{p^n})^{i+1}$ into $H \oplus (\mathbb{Z}_{p^n})^i$ using the same trick.
\end{proof}

For groups whose length is a limit ordinal, we will later show the Scott sentence given in \cref{ScottSentences} is best possible. This means that the sentences we give are best possible for groups of arbitrarily high Scott complexity. 

To prove this, it is helpful to first characterize the back-and-forth relations for abelian $p$-groups. For example, to show that a group $\A$ has no lower Scott complexity than $\Pi_{\alpha+ 1}$, it is enough to exhibit a group $\B \not \cong \A$ such that $\A \equiv_{\alpha} \B$. In this case, $\A$ will have no $d$-$\Sigma_{\alpha}$ Scott sentence since any $d$-$\Sigma_{\alpha}$ sentence true of $\A$ must be true of $\B$. It will be easy to exhibit such a $\B$ if we have a characterization of when the relation $\A \equiv_{\alpha} \B$ holds.

To characterize the back-and-forth relations for reduced Abelian $p$-groups in particular is a longstanding open problem mentioned, e.g., in \cite{AK2000}. Here, we characterize the back-and-forth relations that hold between two arbitrary groups $\A,\B$. 
This is a strengthening of a result of Barker's \cite{Barker1995}, who was able to characterize the back-and-forth relations only for when $\A \cong \B$. 


An exact statement of Barker's results are as follows: 


\begin{proposition}[Barker \cite{Barker1995}]\label{Barker}
    Let $G$ be a reduced abelian $p$-group, and suppose that $\bar a, \bar b \in G^{<\omega}$. Let $g:\bar b \rightarrow \bar a$ map corresponding members of the tuples to one another.
    \begin{enumerate}
        \item $\bar a \leq_{2\alpha} \bar b$ if and only if
        \begin{enumerate}
            \item $g$ extends to an isomorphism $g: \langle \bar b \rangle \to \langle \bar a \rangle$ and
            \item for every $b\in \langle \bar b \rangle$ and $a = g(b)$ we have 
            \begin{equation*}
                h(b) = h(a) < \omega \cdot \alpha \quad \text{ or } \quad h(b), h(a) \geq \omega \cdot \alpha.
            \end{equation*}
        \end{enumerate}
        \item $\bar a \leq_{2\alpha+1} \bar b$ if and only if
        \begin{enumerate}
            \item $g$ extends to an isomorphism $g: \langle \bar b \rangle \to \langle \bar a \rangle$ and
            \item[(b)(i)] If $P_{\omega\cdot \alpha +k}$ is infinite for every $k < \omega$ then for every $b\in \langle \bar b \rangle$ and $a = g(b)$ we have 
            \begin{equation*}
                h(b) = h(a) < \omega \cdot \alpha \quad \text{ or } \quad h(b) \geq \omega \cdot \alpha \text{ and } h(a) \geq \min\{h(b), \omega \cdot \alpha + \omega\}.
            \end{equation*}
            \item[(b)(ii)] If $P_{\omega\cdot \alpha +k}$ is infinite and $P_{\omega \cdot \alpha + k +1}$ is finite, then for every $b\in \langle \bar b \rangle$ and $a = g(b)$ we have 
            \begin{equation*}
                h(b) = h(a) < \omega \cdot \alpha \text{ or } \omega\cdot \alpha \leq h(b) \leq h(a) \leq \omega \cdot \alpha +k \text{ or } h(a) = h(b) > \omega \cdot \alpha +k
            \end{equation*}
            \item[(b)(iii)] If $P_{\omega\cdot \alpha}$ is finite, then for every $b\in \langle \bar b \rangle$ and $a = g(b)$ we have 
            \begin{equation*}
                h(b) = h(a).
            \end{equation*}
        \end{enumerate}
    \end{enumerate}
\end{proposition}

We are now ready to characterize when $(\A, \bar{a}) \leq_\alpha (\B, \bar{b})$ for arbitrary reduced abelian $p$-groups $\A,\B$. Unfortunately, our characterization will not be entirely in algebraic terms. However, we can obtain an entirely algebraic characterization with the additional assumption that 
$|P_{\beta}|= \infty$ for all $\beta$ strictly less than  $\min\{\lambda, \omega \cdot \alpha+\omega\}$ for one of the groups, which is easily seen to be the case, for example, when the last Ulm invariant is infinite or if the length is a limit ordinal. 

\begin{thm}\label{betterbackandforth}
    Let $\A, \B$ be reduced Abelian p-groups of lengths $\lambda_\A$, $\lambda_\B$ respectively. Let $\bar a \in \A^{<\omega}, \bar b \in \B^{<\omega}$ and let $g:\bar b \rightarrow \bar a$ map corresponding members of the tuples to one another.

    \begin{enumerate}
        \item $(\A,\bar a) \leq_{2\alpha+1} (\B,\bar b)$ if and only if
        \begin{enumerate}
            \item for all $\beta < \omega \cdot \alpha$ we have $f_{\A}(\beta) = f_{\B}(\beta)$,
            \item $\lambda_{\A} = \lambda_{\B} < \omega \cdot \alpha$,  or $\lambda_\B \geq \omega \cdot \alpha$  and  $\lambda_\A \geq \min\{\lambda_\B, \omega \cdot \alpha + \omega\}$,
             \item $|P^{\A}_\beta| = \infty$ for all $\beta < min\{ \lambda_{\A}, \omega\cdot\alpha + \omega\}$, or any $\Sigma_1$ sentence with quantifiers relativized to $G_{\omega\cdot\alpha}$ true of $(\B,\bar{b})$ is true of $(\A, \bar{a})$,
            \item $g$ extends to an isomorphism $g: \langle \bar b \rangle \to \langle \bar a \rangle$, and
            \item for every $b\in \langle \bar b \rangle$ and $a = g(b)$ we have 
            \begin{equation*}
                h(a) = h(b) <\omega \cdot \alpha \quad \text{or} \quad \left(h(b) \geq \omega \cdot \alpha \quad \text{and} \quad h(a) \geq \min\{h(b), \omega \cdot \alpha + \omega\}\right).
            \end{equation*}
        \end{enumerate}
        \item $(\A,\bar a) \leq_{2\alpha+2} (\B,\bar b)$ if and only if
        \begin{enumerate}
            \item for all $\beta < \omega \cdot \alpha$ we have $f_{\A}(\beta) = f_{\B}(\beta)$ \\and for $\omega \cdot \alpha \leq \beta <\omega \cdot \alpha + \omega$ we have $f_{\B}(\beta) \leq f_{\A}(\beta)$,
            \item $\lambda_{\A} = \lambda_{\B} < \omega \cdot \alpha + \omega$   or   $\omega \cdot \alpha + \omega \leq \lambda_{\A}, \lambda_\B$,
            \item $|P^{\B}_\beta| = \infty$ for all $\beta < min\{ \lambda_{\B}, \omega\cdot\alpha + \omega\}$, 
            or every $\Sigma_2$ sentence with quantifiers relativized to $G_{\omega\cdot\alpha}$ true of $(\B, \bar{b})$ is true of $(\A,\bar{a})$.
            
           
           
            

            \item $g$ extends to an isomorphism $g: \langle \bar b \rangle \to \langle \bar a \rangle$, and
            \item for every $b\in \langle \bar b \rangle$ and $a = g(b)$ we have 
            \begin{equation*}
                h(b) = h(a) < \omega \cdot \alpha+\omega \quad \text{or} \quad h(b), h(a) \geq \omega \cdot \alpha+\omega.
            \end{equation*}
        \end{enumerate}
        \item $(\A,\bar a) \leq_{\alpha} (\B,\bar b)$ for $\alpha$ a limit ordinal or zero if and only if
        \begin{enumerate}
            \item for all $\beta < \omega \cdot \alpha$ we have $f_{\A}(\beta) = f_{\B}(\beta)$,
            \item $\lambda_{\A} = \lambda_{\B} < \omega \cdot \alpha$   or   $\omega \cdot \alpha \leq \lambda_{\A}, \lambda_\B$,
            \item $g$ extends to an isomorphism $g: \langle \bar b \rangle \to \langle \bar a \rangle$, and
            \item for every $b\in \langle \bar b \rangle$ and $a = g(b)$ we have 
            \begin{equation*}
                h(b) = h(a) < \omega \cdot \alpha \quad \text{or} \quad h(b), h(a) \geq \omega \cdot \alpha.
            \end{equation*}
        \end{enumerate}
    \end{enumerate}
\end{thm}

Before we prove this, as a sanity check, let us pause and note how our results imply Barker's in the case where $\A \cong \B$. It is clear that our conditions in the even and limit cases imply Barker's, and it is also clear how our conditions on the $(2 \alpha + 1)$-relations imply Barker's in the case where $|P_\beta| = \infty$ for all $\beta < min\{ \lambda, \omega\cdot\alpha + \omega\}$. 

Now, let us look at the case where $P_{\omega\cdot\alpha + k + 1}$ is finite. It is known that for any reduced abelian $p$-group $G$, if the subgroup $P$ of elements of order $p$ is finite, then so is the whole group $G$. Thus, in this case we have that $G_{\omega\cdot\alpha + k + 1}$ is finite as well. We can then produce a relativized $\Sigma_1$ sentence that says $b$ has height less than or equal to $\omega\cdot\alpha + k$. It says $b$ cannot be any one of the $n$ elements of height higher than $\omega\cdot\alpha +k + 1$. Let $n$ be the number of elements in $G_{\omega\cdot\alpha + k + 1}$:

$$\exists x_1 \neq x_2 \neq \cdots \neq x_n \in G_{\omega\cdot\alpha + k + 1} \bigwedge_{i \leq n}  b \neq x_i $$

A similar $\Sigma_1$ sentence to that above says $b$ has height less than or equal to $\omega\cdot\alpha + k + n$, for $n \geq 1$. In addition, since $G_{\omega\cdot\alpha + k +n}$ is definable by a relativized $\Sigma_1$ formula, we can express that $b$ has height greater than or equal to $\omega\cdot\alpha + k + n$ for every $n \geq 1$ with the same complexity. Thus it is relativized $\Sigma_1$ to specify exactly the height of $b$ when it is greater than $\omega\cdot\alpha + k$ and less than $\omega\cdot\alpha + \omega$. When $P_{\omega\cdot\alpha}$ is finite, a similar arguments shows we can express that $b$ has height $\omega\cdot\alpha + n$ for every $n \in \omega$.

Now we claim next that in addition, the length of the group is no more than $\omega\cdot\alpha + \omega$. This would mean that $0$ is the only element of height above $\omega\cdot\alpha + \omega$. For this, just note that since $G_{\omega\cdot\alpha + k + 1}$ is finite and the $G_\beta$'s form a decreasing chain, the chain must stabilize before $G_{\omega\cdot\alpha + \omega}$.  

These considerations together show that our results imply Barker's in every case. 

\begin{proof} (Theorem \ref{betterbackandforth})

$(\Rightarrow)$ Suppose that the relevant back-and-forth conditions hold.

Recall from Theorem \ref{ScottSentences}, that the formula expressing that the $\beta$'th Ulm invariant is at least $m$ is $\Sigma_{2 \beta' + 2}$ where $\beta = \omega \cdot \beta' + r$. Now, if $\beta < \omega \cdot \alpha$ then $\beta' < \alpha$ which implies $\beta' + 1 \leq \alpha$ so $2 \beta' + 2 < 2 \alpha + 1$. Thus any formula which is $\Sigma_{2 \beta' + 2}$ is both $\Pi_{2 \alpha + 1}$ and $\Sigma_{2 \alpha + 1}$. Because of this, if $(\A,\bar{a}) \leq_{2 \alpha + 1} (\B, \bar{b})$ then condition (1a) holds. If $\alpha$ is a limit ordinal then $2 \beta' < 2 \alpha = \alpha$ so $2 \beta' + 2 < \alpha$ and thus any formula which is $\Sigma_{2 \beta' + 2}$ is both $\Pi_{\alpha}$ and $\Sigma_{\alpha}$. Therefore, if $(\A, \bar{a}) \leq_\alpha (\B, \bar{b})$ then (3a) holds. If $\omega \cdot \alpha \leq \beta \leq \omega \cdot \alpha + \omega$ then $\beta' = \alpha$ and thus if $(\A,\bar a) \leq_{2\alpha+2} (\B,\bar b)$ condition (2a) holds. 

Again, recall from Theorem \ref{ScottSentences} that the formula expressing the fact that a group has length $\lambda = \omega \gamma + r$ is $d$-$\Sigma_{2 \gamma + 1}$. Thus in cases (1) and (3), if $\lambda_{\A}$ or $\lambda_{\B} < \omega\cdot\alpha$ then $\gamma < \alpha$. Thus $2 \gamma + 1 < 2 \alpha + 1$ (or $2 \gamma + 1 < \alpha$ if $\alpha$ is limit), so $\lambda_{\A} = \lambda_{\B}$. In case (2), if $\lambda_{\A}$ or $\lambda_{\B} < \omega (\alpha+1)$ then $\gamma < \alpha + 1$ and $2 \gamma + 1 < 2 \gamma + 2$ and again $\lambda_{\A} = \lambda_{\B}$. Thus in cases (2) and (3), conditions (2b) and (3b) hold. It remains to show in case (1) that $\lambda_{\A} \geq \lambda_{\B}$ if $\lambda_{\A}, \lambda_{\B} < \omega\cdot\alpha + \omega$. However, if $G_{\omega\cdot\alpha + k} \neq \{0\}$ then $\lambda > \omega\cdot\alpha +k$. Recall that $G_{\omega\cdot\alpha + k}$ is $\Sigma_{2 \alpha+1}$, so the formula expressing the fact that there is some nonzero element in $G_{\omega\cdot\alpha + k}$ is as well.

Since $G_{\omega\cdot\alpha}$ is $\Pi_{2 \alpha}$-definable, any $\Sigma_1$ formula with quantifiers relativized to this subgroup is $\Sigma_{2 \alpha + 1}$ and any $\Sigma_2$ formula with relatives quantifiers is $\Sigma_{2 \alpha + 2}$. Thus in cases (1) and (2), conditions (1c) and (2c) hold. 

In each case, conditions (1d), (2d), and (3c) clearly hold. 

If $h(a)$ or $h(b) < \omega\cdot\alpha$ then WLOG $h(a) = \omega \gamma + r$ for some $\gamma < \alpha$. However, the formula $x \in G_{\omega \gamma + r} \land x \not \in G_{\omega \gamma + r}$ is $d$-$\Sigma_{2 \gamma +1}$. But $2 \gamma + 1 < 2 \alpha + 1$ and, if $\alpha$ is limit, $2 \gamma + 1 < \alpha$. Thus in cases (1) and (3) we have that $h(a) = h(b)$. If $h(a)$ or $h(b) < \omega\cdot\alpha + 1$ then $\gamma < \alpha + 1$ and similarly in case (2) we have that $h(a) = h(b)$. It remains to show in case (1) that if $h(a), h(b) \leq \omega\cdot\alpha + \omega$ then $h(a) \geq h(b)$. The formula $x \in G_{\omega\cdot\alpha + k}$ expresses that $h(x) \geq \omega\cdot\alpha + k$ for an element $x$, however, and is $\Sigma_{2 \alpha +1}$.

   $(\Leftarrow)$ We show that the relevant conditions imply that $(\A,\bar a) \leq_{\gamma} (\B,\bar b)$ by induction on $\gamma$. When $\gamma = 0$, conditions 1(c), 2(d), and 3(c) guarantee that the relevant back and forth relations hold. For $\gamma > 0$, we proceed in cases depending on whether $\gamma$ is odd, even, or a limit ordinal. 

\textbf{Case 1:} Suppose that $\gamma = 2 \alpha + 1$, and assume that $(\A, \bar a), (\B, \bar b)$ satisfy the above conditions for part 1. We wish to show that for any $\bar d \in \B^{<\omega}$ we can find a $\bar c \in \A^{<\omega}$ such that $(\B, \bar b\bar d) \leq_{2\alpha} (\A, \bar a\bar c)$. By induction it is enough to show that the relevant conditions we claim are equivalent to the back and forth relations hold. 

That is, in the case that $\alpha = \delta + 1$, by induction and the fact that $\omega \cdot \delta + \omega = \omega \cdot \alpha$ it is enough to show the following:

  \begin{enumerate}
        \item[(i)] for all $\beta < \omega \cdot \delta$ we have $f_{\A}(\beta) = f_{\B}(\beta)$\\ and for $\omega \cdot \delta \leq \beta <\omega \cdot \alpha$ we have $f_{\B}(\beta) \leq f_{\A}(\beta)$,
        \item[(ii)] $\lambda_{\A} = \lambda_{\B} < \omega \cdot \alpha$   or   $\omega \cdot \alpha \leq \lambda_{\A}, \lambda_\B$,
        \item[(iii)] $|P^{\A}_\beta| = \infty$ for all $\beta < min\{ \lambda_A, \omega\cdot\alpha \}$
        \item[(iv)] $g$ extends to an isomorphism $g: \langle\bar b \bar{d} \rangle \to \langle  \bar a \bar{c}  \rangle$
        \item[(v)] for every $b \in \langle \bar b \bar{d} \rangle$ and $a = g(b)$ we have 
        \begin{equation*}
            h(b) = h(a) < \omega \cdot \alpha \quad \text{or} \quad h(b), h(a) \geq \omega \cdot \alpha.
        \end{equation*}
    \end{enumerate}
Note that (i) follows immediately by assumption (1a), and (ii) follows immediately by assumption (1b).

We now claim condition (iii) is met. If $\lambda_\A < \omega \cdot \alpha$ then 
by assumption (1c) the claim follows.  Using a result of Kaplansky's (\cite{Kaplansky1954} Problem 36) if $\lambda_\A \geq \omega \cdot \alpha$ then $|P^{\A}_\beta| = \infty$ for all $\beta < \omega \cdot \alpha$.

In the case that $\alpha$ is a limit ordinal, we must show that:

\begin{enumerate}
    \item[(a)] for all $\beta < \omega \cdot \alpha$ we have $f_{\A}(\beta) = f_{\B}(\beta)$, which follows from assumption (1a)
    \item[(b)] conditions on $g$, $\bar{b}\bar{d}$, $\bar{a}\bar{c}$ identical to (iv) and (v) above, as well as a condition identical to (ii) above.
\end{enumerate}

Thus, whether or not $\alpha$ is a limit ordinal, it remains only to show (iv), (v) above hold.  

Let $A = \langle \bar a \rangle, B = \langle \bar b \rangle$ be subgroups of $\A, \B$ respectively. Given a $\bar d \in \B$ we seek to extend the isomorphism $g:B \to A$ to an isomorphism $g: B'=\langle \bar b \bar d \rangle \to C'=\langle \bar a \bar c \rangle$ for a suitable $\bar c \in \A$ such that it meets (v) above. We will do this by making a sequence of extensions $B=B_0 \subseteq B_1 \subseteq \cdots \subseteq B'$ such that for each $i$, $x \in B'\setminus B_i$ with $px \in B_i$, and at each stage extending $g$ to $B_{i+1} = \langle x , B_i \rangle$. 

\textbf{Case 1(a):} If $\lambda_{\A} = \lambda_\B < \omega \cdot \alpha$ then our assumption of condition 1(a) forces the groups to be isomorphic by Ulm's theorem. By Barker's results, since $g$ is height-preserving, the tuples $\bar{a},\bar{b}$ are actually automorphic and so we can certainly extend $g$ as needed.

\textbf{Case 1(b):} For this case, assume that $\lambda_{\A}, \lambda_{\B} \geq \omega \cdot \alpha$ and that $|P^{\A}_\beta| = \infty$ for all $\beta < min\{ \lambda_{\A}, \omega\cdot\alpha + \omega\}$.

Let $N_0 = |B'\setminus B|$. We shall ensure that we extend $g$ to $B_i$ such that for all $b \in B_i$ we have $h(a) = h(b) < \omega \cdot \alpha$ or $\omega \cdot \alpha \leq h(b) \leq h(a)$. If this is accomplished for all $i \leq N_0$ then we are done. We shall call the above condition $\star_i$, and, by our assumption of condition 1(d), $g$ meets $\star_0$.
    
    Now by elongating and reordering the tuple $\bar{d}$ if necessary we may assume for all $i$ that if $x \in B'\setminus B_i$ then $px \in B_i$, $x$ is proper with respect to $B_i$, and $h(px)$ is maximal amongst all such elements of $x+ B_i$. Suppose now that we are at the stage $i$ where $g, A_i$ has been constructed as desired and let $g(px) =y$.

    \textbf{Case 1(b)(i):} If $h(y) = h(px) = h(x)+1 < \omega \cdot \alpha$ then we choose a $w$ such that $pw=y$ and $h(w)=h(x)$. Since $h(y) = h(x) +1$ we can choose such a $w$ with $h(w) \geq h(x)$; such a $w$ will have height exactly $h(x)$ since otherwise $h(y) = h(pw) > h(w)$ would be greater than $h(x) + 1$. We know that $w \not \in A_i$ because if $w=g(z)$ for some $z \in A_i$ then $pz=px$ since $g$ is a group isomorphism. This would mean that $h(px-pz) = h(0) = \infty > h(px)$ which contradicts our choice of $x$.
    
    We also get that $w$ is proper with respect to $A_i$, for if there were some $a \in A_i$ such that $h(w+a) \geq h(w)+1 = h(x)+1$ and $a=g(b)$ for $b \in B_i$, then since $w+a \neq 0$ we must have $h(p(w+a)) \geq h(x)+2$ which forces $h(p(x+b)) \geq h(x)+2$. This also contradicts the maximal height of $px \in B_i$.

    We can now extend $g$ to $B_{i+1} = \langle x, B_i \rangle$ by mapping $g(rx+b) = rw+a$ for $0<r<p, b \in B_i$ and $a=g(b)$. 
    
 To see that $g$ meets condition $\star_{i+1}$ take an arbitrary element $rx+b$ of $B_{i+1}$ where $b \in B_i$ and $0 \leq r < p$ since $px \in B_i$.
 Then since $h(w) = h(x)$ and $h(a) \geq h(b)$
 $$ h(rw+a) = min\{ h(w), h(a) \} \geq min\{h(x), h(b) \} = h(rx+b).$$

\textbf{Case 1(b)(ii):} A similar and slightly stronger version of this case is proven in case 2(b)(ii).

\textbf{Case 1(b)(iii):} If $h(x) \geq \omega \cdot \alpha$ then $h(x) < h(px) \leq h(y)$. Thus $y \in G_{h(x) + 1}^{\A}$ so there exists $w_1 \in \A$ such that $pw_1 = y$ and $h(w_1) \geq h(x)$. 

Now $G_{h(x)}^{\A}$ is nonempty so $P_{h(x)}^{\A}$ is infinite. Thus there exists an element $w_2 \in P_{h(x)}^{\A} -  A_i$. Let $w = w_1 + w_2$ and extend $g$ to $B_{i+1}$ by defining $g(x) = w$. It is easy to check that the extension remains an isomorphism since $pw = y$ and $w \not \in A_i$.

To see that $g$ meets condition $\star_{i+1}$ take an arbitrary element $rx+b$ of $B_{i+1}$ where $b \in B_i$ and $0 \leq r < p$ since $px \in B_i$. We have that $g(rx+b) = rw+a$. If $h(b) < \omega \cdot \alpha$ then $h(rx+b) = h(b)$ and $h(b) = h(a) < \omega \cdot \alpha$ by the fact that $g$ satisfies $\star_i$. Now $h(rw) \geq min\{h(w_1), h(w_2) \} \geq h(x) \geq \omega \cdot \alpha$ so $h(rw+a) = h(a) = h(b) = h(rx+b)$. Thus $g$ preserves heights below $\omega \cdot \alpha$.

If $h(b) \geq \omega \cdot \alpha$ then $h(a) \geq h(b)$. Moreover,
$$h(rw+a) \geq min\{ h(w), h(a)\} \geq min\{ h(w_1), h(w_2), h(a) \} \geq min\{ h(x), h(b)\}$$
However, since $x$ is proper with respect to $B_i$ we have that $h(rx+b) = min\{h(x), h(b)\}$. Thus $h(rw+a) \geq h(rx + b)$, as desired. Thus $g$ is non-decreasing for heights above $\omega \cdot \alpha$ as well, and condition $\star_{i+1}$ is met.

\textbf{Case 1(c):} For this case, assume that $\lambda_{\A}, \lambda_{\B} \geq \omega \cdot \alpha$ and that for some $\beta < min\{ \lambda_{\A}, \omega\cdot\alpha + \omega\}$, $|P^{\A}_\beta| \neq \infty$. Then we must have that $\omega\cdot\alpha \leq \lambda_{\B} \leq \lambda_{\A} \leq \omega\cdot\alpha + \omega$. As in case 1(b), we will extend $g$ to $B_i$ such that condition $\star_i$ holds, and may assume that for all $i$, if $x \in B'\setminus B_i$ then $px \in B_i$, $x$ is proper with respect to $B_i$, and $h(px)$ is maximal amongst all such elements of $x+ B_i$.

We now claim that we can add the additional simplifying assumption that $h(x) \leq \omega \cdot \alpha$ for any $x \in B'-B_0$. Note that the $B_i$'s can be built before the $A_i$'s and the associated extensions of $g$. The goal is to deal first with all of the elements in $B' \cap G^{\B}_{\omega\cdot\alpha}$ before we begin the construction. Let $\bar{d}'$ be the sub-tuple of $\bar{d}$ of elements in $G^{\B}_{\omega\cdot\alpha}$. We claim that we can extend $g$ to a height non-decreasing isomorphism. 

    Now, by (1c) any $\Sigma_1$ formula where the quantifiers are relativized to $G_{\omega\cdot\alpha}$ true of $\bar{b}$ is true of $\bar{a}$. Let, for any group $G$ and tuple $\bar{s} \in G$, the formula $\phi^G_{\bar{s}}(\bar{x})$ define the full atomic type of $\bar{s}$ in $G$. Note this is $\Delta_0$ since the language of groups is finite. Now consider the formula 
    
    $$\exists \bar{u} \in G_{\omega\cdot\alpha}(\phi_{\bar{b}\bar{d}'}^{B}(\bar{b}\bar{u}) \land \bigwedge_{i = 0}^{|\bar{d}'|-1} u_i \in G_{h(d_i')} )$$

    Note that $h(d_i')$ is strictly less than $\alpha + \omega$ since $\omega\cdot\alpha \leq \lambda_{\B} \leq \lambda_{\A} \leq \omega\cdot\alpha + \omega$. Thus $G_{h(d_i')}$ is $\Sigma_1$-definable in $G_{\omega\cdot\alpha}$ so is $\Sigma_1$-definable in $G$ with quantifiers relativized to $G_{\omega\cdot\alpha}$. 

    This formula is true of $\bar{b}$ so is true of $\bar{a}$. If $\bar{c'}$ is the witness for this formula in $\A$, then $\bar{a}\bar{c'}$ will have the same atomic type as $\bar{b}\bar{d}'$, and the map sending the corresponding tuples to each other is height non-decreasing. In other words,

$$ ( \A, \bar{a} \bar{c}') \equiv_0 (\B, \bar{b} \bar{d}') $$

We now consider the associated extension of $g$ which includes $\bar{d}'$ in its domain. Replace $B_0$, $A_0$ with the groups generated by the domain and range, respectively, of this extension. Since the extension is height non-decreasing above $\omega\cdot\alpha$, condition $\star_0$ is met. 

However, we need to show that our simplifying conditions still hold if we reorder the tuple $\bar{d}$ so that all of the elements of height $\geq \omega\cdot\alpha$ appear first, generating $B_0$. It is easily seen that we may continue to assume that $px \in B_i$. 

Now if $x$ is proper with respect to $B_i$, with $h(x) < \omega\cdot\alpha$ and we consider $B_i' = \langle B_i, z \rangle$ where $z$ has height $\geq \omega\cdot\alpha$. Then for arbitrary $b \in B_i$ and $n \in \omega$, we have that 
$h(x) \geq h(x + b)$ so $h(x+b) < \omega\cdot\alpha \leq h(nz)$. This means that
$$ h(x + b + nz) = h( (x+b) + nz ) = h(x+b) \leq h(x)$$
so $x$ remains proper with respect to $B_i'$. This argument generalizes to show that after reordering the tuple we may continue to assume that $x$ is proper with respect to $B_i$. 

Finally, note that the simplifying assumption that $h(px)$ is maximal is achieved by possibly replacing $x$ with an element of the same height as $x$ so this may still be assumed after reordering the tuple. 

    Let $g(px) = y$. 

    Since $g$ meets condition $\star_i$, we know that either $h(px)=h(y) <\omega\cdot \alpha$ or $h(px) \geq \omega \cdot \alpha$ and $h(y) \geq h(px)$.

Now we continue with the construction. Assume $A_i$ and $g: B_i \rightarrow A_i$ has been constructed and condition $\star_i$ is met.

\textbf{Case 1(c)(i):} In the case where $h(y) = h(px) = h(x) + 1 < \omega \cdot \alpha$, we proceed as in Case 1(b)(i). 

\textbf{Case 1(c)(ii):} In the case where $h(x) = \gamma < \omega \cdot \alpha, h(x) > \gamma + 1$ we proceed as in  Case 1(b)(ii). 

\textbf{Case 2:} 
Suppose that $\gamma = 2 \alpha + 2$, and assume that $(\A, \bar a), (\B, \bar b)$ satisfy the above conditions for part 2. Now we must do something very similar to before to show that given any $\bar d$ we can find a $\bar c$ such that $(\B, \bar b\bar d), (\A, \bar a \bar c)$ meet the conditions for part 1. As before, most of the conditions follow by assumption. 

Our remaining concern is extending the isomorphism $g$ to a map $\langle \bar b, \bar d \rangle \to \langle \bar a, \bar c \rangle$. Let $\bar d \in \B^{<\omega}$ and we once again let $B_0 = B = \langle \bar b \rangle$, $B' = \langle \bar b \bar d \rangle$, and $A_0 = A = \langle \bar a \rangle$.

    \textbf{Case 2(a):} If $\lambda_\A = \lambda_\B < \omega \cdot \alpha$ then, as in case 1(a), we conclude that $\A$ and $\B$ are isomorphic.

  \textbf{Case 2(b):} For this case, assume that $\lambda_{\A}, \lambda_{\B} \geq \omega \cdot \alpha$ and that $|P^B_{\beta}| = \infty$ for all $\beta < min\{\lambda_{\A}, \omega \cdot \alpha + \omega \}$.
   Let
    \begin{equation*}
        M = \max\{m: b' \in B', h(b') = \omega\cdot \alpha + m \text{ for } m\in \omega\} +1,
    \end{equation*}
    \begin{equation*}
        N_0 = |B'/B|+M, \text{ and } N_{i+1} = N_i-1.
    \end{equation*}
    For each $0\leq i \leq N_0$ we would like to extend $g$ to $B_i$ in such a way that it satisfies
    \begin{equation*}
        h(b) = h(a) < \omega \cdot \alpha +M\quad \text{or} \quad \left( h(b) \geq \omega \cdot \alpha+\omega \quad \text{and} \quad h(a) \geq \omega \cdot \alpha+N_i \right)
    \end{equation*}

    \noindent for all $b \in B_i$, $a=g(b)$. We will call this condition $\star_i$. If this is met for all $i$ then $g$ will satisfy
    \begin{equation*}
        h(b') = h(a') \leq \omega \cdot \alpha \quad \text{or} \quad h(a') \geq \omega \cdot \alpha \text{ and } h(b') \geq \min\{h(a'), \omega \cdot \alpha + \omega\}
    \end{equation*}

    \noindent for all $b' \in B'$, which would complete our proof that $(\B, \bar b \bar d) \leq_{2\alpha+1} (\A, \bar a \bar c)$. By assumption, $g$ meets condition $\star_0$.


Now by induction on $i$, we take an $x \in B'\setminus B_i$ such that $px \in B_i$, $x$ is proper with respect to $B_i$, and $h(px)$ is maximal amongst all such elements of $x+ B_i$. Let $g(px) =y$. Since $g$ meets condition $\star_i$ we know that either $h(px)=h(y) <\omega\cdot \alpha+M$ or $h(px) \geq \omega \cdot \alpha+\omega$ and $h(y) \geq \omega \cdot \alpha + N_i$.
    
    \textbf{Case 2(b)(i):} If $h(y) = h(px) = h(x)+1 < \omega \cdot \alpha+M$ then we choose a $w$ such that $pw=y$ and $h(w)=h(x)$. Since $h(y) = h(x) +1$ we can choose such a $w$ with $h(w) \geq h(x)$; such a $w$ will have height exactly $h(x)$ since otherwise $h(y) = h(pw) > h(w)$ would be greater than $h(x) + 1$. 
    
    We know that $w \not \in A_i$ because if $w=g(z)$ for some $z \in A_i$ then $pz=px$ since $g$ is a group isomorphism. This would mean that $h(px-pz) = h(0) = \infty > h(px)$ which contradicts our choice of $x$.
    
    We also get that $w$ is proper with respect to $A_i$, for if there were some $a \in A_i$ such that $h(w+a) \geq h(w)+1 = h(x)+1$ and $a=g(b)$ for $b \in B_i$, then since $w+a \neq 0$ we must have $h(p(w+a)) \geq h(x)+2$ which forces $h(p(x+b)) \geq h(x)+2$. This also contradicts the maximal height of $px \in B_i$.

    We can now extend $g$ to $B_{i+1} = \langle x, B_i \rangle$ by mapping $g(rx+b) = rw+a$ for $0<r<p, b \in B_i$ and $a=g(b)$. 
    
To see that $g$ meets condition $\star_{i+1}$ take an arbitrary element $rx+b$ of $B_{i+1}$ where $b \in B_i$ and $0 < r < p$ since $px \in B_i$.
 Then since $x, w$ are proper, and $h(w) = h(x)$ we have for $r \neq 0, r < p$
 $$ h(rw+a) = min\{ h(w), h(a) \} = min\{h(x), h(b) \} = h(rx+b).$$
 for if $h(a)>h(w)$ then $min\{ h(w), h(a) \} = h(w)<\omega\cdot \alpha +M$, and if $h(w)\geq h(a)$ then $h(a) < \omega \cdot \alpha + M$ and so in both cases $h(a) = h(b)$. If $r=0$ in $rw+a$ then since $g$ already meets $\star_{i}$ so also does our extension of $g$. 
 
\textbf{Case 2(b)(ii):} Suppose $h(px) < \omega\cdot\alpha + M$ and let $h(x) = \gamma$ with $h(px) > \gamma + 1$. Since $h(px) > \gamma +1$, there is a $v \in \B_{\gamma+1}$ such that $pv=px$.  The element $x-v$ is then in $P^{\B}_{\gamma}$ and it also has height $\gamma$ and is proper with respect to $B_i$. We can apply Lemma \ref{Ulmiso} to see that the range of $u$ is not all of $P^{\B}_{\gamma}/P^{\B}_{\gamma+1}$.

    Now, by assumption and induction on $i$, $g$ preserves heights for elements of height $< \omega \cdot \alpha + M$ and is an isomorphism between $A_i$ and $B_i$. Therefore $g$ maps $B_i \cap \B_{\gamma+1}$ in bijection with $A_i \cap \A_{\gamma+1}$, $B_i \cap \B_\gamma$ in bijection with $A_i \cap \A_\gamma$ and $B_i \cap \B_\gamma \cap p^{-1}\B_{\gamma+2}$ in bijection with $A_i \cap \A_\gamma \cap p^{-1}\A_{\gamma+2}$. 
    
    Since $u$ is not surjective and $(B_i \cap \B_\gamma \cap p^{-1}\B_{\gamma+2})/(B_i \cap \B_{\gamma+1})$ is finite, its dimension as a vector space over $\mathbb{Z}_p$ is strictly less than $f_{\B}(\gamma).$
   Since $(A_i \cap \A_\gamma \cap p^{-1}\A_{\gamma+2})/(A_i \cap \A_{\gamma+1})$ has the same cardinality and thus the same dimension since it is finite, its dimension is also less than $f_{\B}(\gamma) \leq f_{\A}(\gamma)$. Thus we can apply Lemma \ref{Ulmiso} in reverse, giving us an element $w_1 \in \A$ such that $pw_1=0$, $h(w_1)=\gamma$, and which is proper with respect to $A_i$.

    Since $h(y) > \gamma+1$ we know that $y=pw_2$ for $w_2 \in \A_{\gamma+1}$. Let $w=w_1+w_2$. Then $pw= 0+y$, $h(w)=h(w_1)=\gamma$, and $w$ is proper with respect to $A_i$. If we extend $g$ by $g(x)=w$ as before, we see that $g$ satisfies $\star_{i+1}$ since $w$ was chosen to be proper as in the previous case.
    
\textbf{Case 2(b)(iii):} $h(px) \geq \omega \cdot \alpha + \omega$ and $h(y) \geq \omega \cdot \alpha + N_i$. As in case 1(b)(iii), we can find a $w_1 \in \A$ such that $pw_1=y$ and $h(w_1) = h(y)-1 \geq \omega \cdot \alpha +N_i - 1$. Furthermore, if we are in this case, then $\lambda_{\B} \geq \omega \cdot \alpha + \omega$ which forces $\lambda_{\A} \geq \omega \cdot \alpha + \omega$. Thus, since $P^{\A}_{\omega \cdot \alpha+N_i}$ is infinite, we can find a $w_2 \in P^{\A}_{\omega \cdot \alpha+N_i}\setminus A_i$. Set $w = w_1 +w_2$ and extend $g$ by sending $x$ to $w$. Note that $w \not \in A_i$ and $h(w) \geq \omega \cdot \alpha + N_{i+1} > \omega \cdot \alpha +M$.

Let $b \in B_{i}$. If $h(b) \leq \omega \cdot \alpha +M$ then by assumption $h(a)=h(b)$ and so for $0\leq r < p$, $h(rx+b) = h(b)$ and $h(rw+a) = h(a)$. If $h(b) \geq \omega \cdot \alpha + \omega$ then by induction on $i$, $h(a)\geq \omega \cdot \alpha+N_i> \omega \cdot \alpha +N_{i+1}$. Thus, $h(rx+b) \geq \omega \cdot \alpha+\omega$ and $h(rw+a) \geq \min\{h(w), h(a)\} \geq \omega \cdot \alpha + N_{i+1}$. This shows that $g$ meets condition $\star_{i+1}$.

\textbf{Case 2(c):} For this case, assume that $\lambda_{\A}, \lambda_{\B} \geq \omega \cdot \alpha$ and that every $\Sigma_2$ sentence with quantifiers relativized to $G_{\omega\cdot\alpha}$ true of $(B, \bar{b})$ is true of $(\A,\bar{a})$. Then we show first that as in Case 1 we may assume that $h(x) < \omega \cdot \alpha.$ We do this just as before by reordering the tuple $\bar{d}$ so that the sub-tuple $\bar{d}'$ of elements of height greater than $\omega \cdot \alpha$ appears first. 
    
We claim that we can extend $g$ to a height-preserving isomorphism including $\bar{d}'$. Consider the formula

$$\exists \bar{u} \in G_{\omega\cdot\alpha}(\phi_{\bar{b}\bar{d}'}^{B}(\bar{b}\bar{u}) \land \bigwedge_{i = 0}^{|\bar{d}'|-1} u_i \in G_{h(d_i')} \land u_i \not \in G_{h(d_i') + 1})$$

This is a $\Sigma_2$ formula true of $\bar{b}$ so is true of $\bar{a}$. If $\bar{c'}$ is the witness for this formula in $\A$, then $\bar{a}\bar{c'}$ will have the same atomic type as $\bar{b}\bar{d}'$, and the map sending the corresponding tuples to each other is height preserving. In other words,

$$ ( \A, \bar{a} \bar{c}') \equiv_0 (\B, \bar{b} \bar{d}') $$

We now consider the associated extension of $g$ which includes $\bar{d}'$ in its domain. Replace $B_0$, $A_0$ with the groups generated by the domain and range, respectively, of this extension. As before, we may assume that $px \in B_i$, $x$ is proper with respect to $B_i$, and $h(px)$ is maximal amongst all such elements of $x+ B_i$.
    
Assume that we have built $g: B_i \rightarrow A_i$ such that $h(b) = h(a)$ for all $b \in B_i$, $a=g(b)$. We would like to extend $g$ to $B_{i+1}$ such that $g$ preserves heights. We will call this condition $\star_{i+1}$.




\textbf{Case 2(c)(i):} We extend $g$ as in case 1(b)(i). Since we begin with a height-preserving $g$ and $w$ is proper with respect to $A_i$, our extension of $g$ will be height-preserving. 

\textbf{Case 2(c)(ii):} We extend $g$ as case 2(b)(ii). Since we begin with a height-preserving $g$ and $w$ is proper with respect to $A_i$, our extension of $g$ will be height-preserving.

 \textbf{Case 3:} Assume that $(\A,\bar a), (\B, \bar b)$ meet the conditions of 3 and let $\bar d \in \B^{<\omega}$ and $\beta < \alpha$.

    \textbf{Case 3 (a):} If $\lambda_{\A} = \lambda_{\B} < \omega \cdot \alpha$ then as before the groups are isomorphic.

    \textbf{Case 3 (b):} If $\lambda_{\A}, \lambda_{\B} \geq \omega \cdot \alpha$ then we know that $|P^\B_{\omega \cdot \beta + k}|= \infty$ (Kaplansky \cite{Kaplansky1954} Problem 36) and so we meet the conditions for $(\A, \bar a) \leq_{\beta +1} (\B, \bar b)$ regardless of whether $\beta$ is odd, even, or a limit ordinal. Hence, by induction there exists a $\bar c \in \A^{<\omega}$ such that $(\B, \bar b \bar d) \leq_{\beta} (\A, \bar a \bar c)$ for all $\beta < \omega\cdot\alpha$.
\end{proof}

Using the characterization of the back-and-forth relations, we can show that the Scott sentences we give in Theorem \ref{ScottSentences} are optimal for arbitrarily high $\alpha$.

\begin{thm}\label{Scottcomplex}
    Suppose that $\A$ is a structure with $\lambda_\A = \omega \cdot \alpha$ where $\alpha$ is a limit ordinal. Then $\A$ has Scott complexity $\Pi_{2 \alpha + 1}$. 
\end{thm}

\begin{proof}
    By Theorem \ref{ScottSentences}, $\A$ has a Scott sentence of complexity $\Pi_{2 \alpha + 1}$. It remains to show that $\A$ has no lower-complexity Scott sentence. To this end, it is enough to exhibit $\A \not \cong \B$ such that $\A \equiv_{2 \alpha} \B$. If so, then it is impossible for $\A$ to have a $d$-$\Sigma_{2 \alpha}$ Scott sentence. 

    Consider $\B$ of strictly longer length than $\A$, such that the Ulm invariants of both structures agree below $\omega \cdot \alpha$. It is known that any Ulm sequence with infinitely many nonzero Ulm invariants ``between" each relevant pair of limit ordinals is realizable as the Ulm sequence of a reduced abelian $p$-group (see e.g. \cite{Kaplansky1954}) so such a group exists. Then $\A \equiv_{2 \alpha} \B$ by \Cref{betterbackandforth}, part (3). 
\end{proof}


\section{Open Questions and Problems}

We were able to prove that when $G$ is finite-length and has at least two infinite Ulm invariants that our upper bounds on Scott sentences were best possible. One can ask the same question about $G$ of arbitrary length:

\begin{question}
     Are the Scott sentences given in \Cref{ScottSentences} optimal in every case where there are at least two infinite Ulm invariants?
\end{question}

It is often of interest whether certain kinds of Scott complexities are realizable in a class of structures. In particular, it difficult to construct an example of structure of Scott complexity $\Sigma_{\lambda + 1}$ where $\lambda$ is a limit ordinal. Our examples only show that the Scott complexity $\Pi_\alpha$ is realizable for various $\alpha$. 

\begin{question}
     Are the Scott complexities $\Sigma_\alpha$ or $d$-$\Sigma_\alpha$ realizable within the class of abelian $p$-groups? If so, for which $\alpha$?
 \end{question}

Finally, it should be noted that the best characterization of the back-and-forth relations, an open problem from \cite{AK2000}, should be an entirely algebraic characterization.

\begin{question}
    Can one give a completely algebraic characterization of the back-and-forth relations on abelian $p$-groups?
\end{question}

\printbibliography
\end{document}